\documentclass[imslayout,noinfoline,a4paper,12pt,twoside]{imsart}
\RequirePackage[OT1]{fontenc} 
\RequirePackage{amsthm,amsmath,amssymb,natbib} 
\startlocaldefs
\theoremstyle{plain}
\newtheorem{lem}{Lemma}%[section]
\endlocaldefs
\renewcommand{\[}{\begin{eqnarray*}}
\renewcommand{\]}{\end{eqnarray*}}
\newcommand{\la}{\begin{eqnarray}}
\newcommand{\al}{\end{eqnarray}}

\renewcommand{\epsilon}{\varepsilon}
\renewcommand{\phi}{\varphi}

\newcommand{\N}{{\mathbb N}}
\newcommand{\R}{{\mathbb R}}

\renewcommand{\P}{{\mathbb P}}\newcommand{\E}{{\mathbb E}}  

\begin{document}

\begin{frontmatter}

%==========================================================================

\title{Sums of norm spheres are norm shells  and lower triangle 
inequalities are sharp}

\runtitle{ Sums of norm spheres} %{ and sharp lower triangle inequalities }

\begin{aug}
\author{\fnms{Lutz} \snm{Mattner}%\thanksref{t1,t2}
\ead[label=e1]{mattner@uni-trier.de}}%,
%\author{\fnms{Second} \snm{Author}\thanksref{t3}\ead[label=e2]{second@somewhere.com}}
% \and
% \author{\fnms{Uwe} \snm{R\"osler}%\thanksref{t1}
% \ead[label=e2]{roesler@math.uni-kiel.de}}%
% %\ead[label=u1,url]{http://www.foo.com}}

%\thankstext{t1}{Some comment}
%\thankstext{t2}{First supporter of the project}
%\thankstext{t3}{Second supporter of the project}
\runauthor{Lutz Mattner}

\affiliation{Universit\"at Trier} 
% %%
 {\rm \today}\\
 {\footnotesize\tt \jobname.tex}
% %%

\address{Universit\"at Trier\\
Fachbereich IV -  Mathematik\\ 
54286 Trier \\
Germany\\
\printead{e1}}
\end{aug}

 \begin{abstract}
 The statements in the title are explained and proved, as a little
 exercise in elementary normed vector space theory
 at the level of Chapter 5 of Dieudonn\'e's 
 \textit{Foundations of Mathematical Analysis}.
 A connection to recent moment bounds for submartingales is sketched. 
 \end{abstract}


\end{frontmatter}

%===========================================================================

%{\footnotesize \tableofcontents} %

\section{A stochastic introduction: 
From moment bounds for sub\-mar\-tin\-gales 
to elementary normed space theory}
This note treats from Section~\ref{Sec.Old.Intro} onwards an elementary
and natural but seemingly not too well-known exercise in analysis,
suitable for mathematics students after their first year at 
university. The present somewhat lengthy   
but less elementary and  less detailed introductory section explains
our probabilistic motivation for undertaking the exercise,
but it is not logically necessary for understanding the 
remaining two sections.
 
Our motivation stems from moment bounds in probability. 
Referring to the appropriate sections of 
\cite{ChowTeicher},  \cite{delaPenaGine} and \cite{Petrov}
for good general introductions to this topic,
we restrict our attention here to the  somewhat specialized
subtopic of finding the {\em optimal} lower and upper bounds
for the $r$th absolute moment $\E |S_n|^r_{}$, 
with $r\in[1,\infty[$,  of a sum $S_n=\sum_{i=1}^n X_i$
of real-valued random variables, given $n\in\N$, given only the  individual  moments
$\E |X_1|^r_{},\ldots,\E|X_n|^r_{}$ {\em of the same order}, 
and given some %interesting 
structural assumption on the process $(X_1,\ldots,X_n)$. 
Structural assumptions of interest include the following  ones, 
ordered here according to increasing generality,
\[
 \text{IIDC:}&&\hspace{-2ex} \text{$X_1,\ldots,X_n$ 
  are independent, identically distributed, and centred} \\
 \text{IC:}&&\hspace{-2ex} \text{$X_1,\ldots,X_n$ 
  are independent and centred}\\
 \text{MG:}&&\hspace{-2ex} \text{$(S_1,\ldots, S_n) $ is a martingale} 
\]
with ``centred'' meaning $\E X_i = 0$  for each $i$.

For example, under the assumption IIDC,
bounding  $\E |S_n|^r_{}$ in terms of 
$\E |X_1|^r_{}= \ldots =\E |X_n|^r_{}$ is equivalent to 
bounding $ \E |\frac 1n \sum_{i=1}^n Y_i - \E Y_1|^r_{}$ 
for not necessarily centred IID random variables $Y_i$
in terms of $\E |Y_1-\E Y_1|^r_{}$, 
and this is of interest as one reasonable way of assessing the quality of 
$\frac 1n \sum_{i=1}^n Y_i$ as a statistical estimator for $\E Y_1$.  

The best known and most convenient optimal bounds of the kind considered here
refer to the case of $r=2$, namely
\la
   \sum_{i=1}^n \E|X_i|^2_{} &\le& \E |S_n|^2_{} \,\,\le\,\,  \sum_{i=1}^n  \E|X_i|^2_{} 
\al
valid and then  trivially optimal under each of IIDC, IC and MG. 
Further optimal bounds are rare. 
Restricting our attention to $r\in {[1,2[}$
here (for brevity, and since for $r>2$ the %different but 
more effective inequalities of Rosenthal
\cite[pp.~59, 83]{Petrov}
involving  moments of orders $2$ and $r$ are available)
and under the structures considered,
the only ones known to me are the trivial  upper bound   
\la
  \E |S_n| &\le& \sum_{i=1}^n \E|X_i|
\al
valid (by the usual triangle inequality for expectations, without
using any structural assumption) and optimal under each of IID, IC, MG
(as can be seen by considering independent random variables $X_i$ 
with $\P(X_i = - a_i/p)=   \P(X_i = a_i/p) = p/2$ and $\P(X_i=0)=1-p$,
and letting $p\rightarrow 0$),  the lower bound of Hornich type 
\la        \label{Mattner.Hornich}
  \E |S_n|    &\ge&  c^{}_n\,\E|X_1|  \quad \text{ under IIDC}
\al 
where, writing  $\mathrm{b}(n,p;k):=\binom{n}{k}p^k(1-p)^{n-k}$
for  binomial probabilities, we have  put
$c_n := n\, \mathrm{b}(n,\frac{\lfloor n/2 \rfloor}{n}; \lfloor n/2\rfloor )$,
which is $   \sim$ $ \sqrt{\frac{2n}{\pi}}$ for $n\rightarrow \infty$,
and finally
\la \qquad\label{Mattner.Roesler}
  \E |S_n| &\ge&    \max\, \Big\{\E |X_k| -\sum_{i=1}^{k-1} \E |X_i|\Big\}_{k=1}^n \cup 
\Big\{ \frac {\E |X_k|}2\Big\}_{k=3}^n    \quad \text{ under MG}
\al
where $\{b_k\}_{k=n_1}^{n_2}$ denotes the possibly empty set
$\{b_k \,:\, k\in\N, \, n_1\le k \le n_2\}$.
See \cite{Mattner} for \eqref{Mattner.Hornich}
and   \cite{MattnerRoesler} for \eqref{Mattner.Roesler}, and both  for 
more detailed discussions and appropriate references.    

Generally speaking  one should expect that proving   optimal bounds as considered
here, or equivalently solving associated  extremal problems subject to constraints,
should  be the most difficult under IC and easiest under MG,
since $\E |S_n|^r_{}$  is  under MG a {\em linear} functional of the law
of $(S_1,\ldots,S_n)$ which varies in a {\em convex} set,
while under IC it is an {\em $n$-linear}
% integral with repect to $P_1\otimes\ldots\otimes P_n$ 
functional of   the individual laws
$P_1,\ldots,P_n$, namely an integral with respect to the product measure
$P_1\otimes\ldots\otimes P_n$, 
and finally under IIDC it is an integral with respect to a ``power measure'' $P^{\otimes n}$.
It is therefore surprising  and certainly an instance of luck 
that \eqref{Mattner.Hornich} has been proved. 
The paper \cite{MattnerRoesler} resulted from an attempt to find an  analogue
of \eqref{Mattner.Hornich}  under IC. Being unsuccessful in that respect, the
authors  of \cite{MattnerRoesler}  of course 
modified the assumptions and  finally  came up with  \eqref{Mattner.Roesler}
and with a similarly strange  result for submartingales. 
Oddly enough, during the many  pizza evenings  it took them, they never
thought of first solving their moment problems under the simplest of all structural  
assumptions, namely
\[
 \text{N:}&&\hspace{-2ex} \text{No assumption, i.e., $X_1,\ldots,X_n$ are
 arbitrary random variables}
\]
Returning to general $r\in[1,\infty[$ now and considering 
$(\E |S_n|^r)^{1/r}_{}$ instead of $\E |S_n|^r_{}$, 
this amounts to writing down  the best upper and lower bounds for 
$\mathrm{L}^r_{}$~norms of  sums given the norms of the summands,  and 
the nonsurprising result, see \eqref{TriangleIneq2} below, turns out to be the
same more generally on every normed space of dimension at least two. The proof
of the optimality of the lower bound is a  not completely trivial and 
nice exercise in analysis, as we try to show below.

\section{An elementary analytic introduction} \label{Sec.Old.Intro}
Let $(E,\|\cdot\|)$ be a normed vector space over the real numbers $\R$. 
Then, for $x,y\in E$, the triangle inequalities 
\la         \label{TriangleIneq1}
  \big|\, \|x\|-\|y\|\, \big| &\le& \|x+y\| \,\,\le \,\,\|x\|+\|y\|
\al 
provide optimal lower and upper bounds for $\|x+y\|$ in terms
of  $\|x\|$ and $\|y\|$. Namely, ignoring the trivial case where 
$E$ is zero-dimensional and given any positive numbers $a,b\in[0,\infty[$,
we get  $x,y\in E$ with $\|x\| = a$, $\|y\|=b$,  
and  $ \|x+y\| = |a-b|$ or $ \|x+y\| = a+b$,  by choosing
some $u\in E$ with $\|u\| = 1$  and then putting 
$x:= au $ and $y := -bu$ or $y:= bu$, respectively.  

How about more than two vectors? For $n\ge 2$ and $x_1,\ldots,x_n\in E$,
the inequalities \eqref{TriangleIneq1} generalize  quite obviously to 
\la           \label{TriangleIneq2}
   \max_{k=1}^n\Big( \|x_k\| - \sum_{i\neq k} \|x_i\| \Big)^{}_+
  &\le& \Big\|\sum_{i=1}^n x_i\Big\| \,\,\le \,\, \sum_{i=1}^n \|x_i\|
\al
with $c_+:=\max\{c,0\}$ for   $c\in\R$,
and with ``$i\neq k$'' of course meaning 
``$i\in\{1,\ldots,n\}\setminus\{k\}$''.
Here the left hand inequality in  \eqref{TriangleIneq2}
follows from observing 
\[
 \Big\|\sum_{i=1}^n x_i\Big\| &=&
  \Big\|x_k + \sum_{i\neq k} x_i\Big\|
 \,\,\ge \,\, \|x_k\| -  \Big\|\sum_{i\neq k} x_i\Big\|
 \,\,\ge\,\,  \|x_k\| -  \sum_{i\neq k} \|x_i\|
\] 
for each $k\in\{1,\ldots,n\}$.
But can such a simple-minded approach yield  optimal bounds
for $\| \sum_{i=1}^n x_i\|$ given $\|x_1\|,\ldots,\|x_n\|$ also if $n\ge 3$? 
For the upper bound in \eqref{TriangleIneq2}, the answer 
is obviously ``yes'',
by choosing the $x_i$ to be appropriate positive multiples of some nonzero 
$u\in E$, as above for the special case of $n=2$.
For the lower bound, however, the answer is clearly ``no'' when $E$ is 
one-dimensional: Then, for example, if $n=3$ and  
$\|x_1\| = \|x_2\| = \|x_3\| = 1$, we 
necessarily have $\|x_1+x_2+x_3\| \in\{1,3\}$, 
wheras \eqref{TriangleIneq2}
only yields the trivial lower bound $\|x_1+x_2+x_3\| \ge 0$. 
The purpose of the remainder of this note is to show
that in higher dimensions the inequalities \eqref{TriangleIneq2} are indeed 
optimal:

\medskip 
\textsc{Proposition.}
\textsl{Let  $(E,\|\cdot\|)$ be a normed vector space over $\R$
with dimension at least $2$. Then, for every choice of 
$a,b\in [0,\infty[$, we have 
\la          \label{SumsAreShells1}
&&\big\{x+y\,:\, x,y\in E,\, \|x\|=a,\, \|y\|=b\big\}  \\
&& \qquad \qquad\qquad \qquad \qquad\qquad \,\,=\,\,   
\big\{z\in E\,:\, |a-b|\le \|z\| \le a+b\big\}  \nonumber
\al
and, more generally, for every integer $n\ge 2$ and 
every  $a \in {[0,\infty[}_{}^n$,   
\la                  \label{SumsAreShells2}
&& \big\{\sum_{i=1}^n x_i\,:\, x_i\in E \text{ and }\|x_i\|=a_i
 \text{ for each }i \big\}  \\
&& \qquad \qquad \qquad\qquad  =\,\,   
 \big\{ z\in E\,:\,  
 \max_{k=1}^n(a_k-\sum_{i\neq k} a_i)^{}_+ \,\le\, \|z\| \,\le\,
 \sum_{i=1}^n a_i \big\} 
 \nonumber 
\al}

Writing 
\la          \label{Def Sr}
  S_r &:=& \{x\in E\,:\, \|x\| = r\}
\al
for the possibly degenerate norm  sphere of radius $r\in[0,\infty[$ 
around $0\in E$,
identity~\eqref{SumsAreShells1} says that the sum of two 
norm spheres is  the closed norm shell around 0
with inner and outer  radii obtained from the 
inequalities~\eqref{TriangleIneq1}, and identity~\eqref{SumsAreShells2}
is the corresponding statement for the sum of $n$ spheres,
with the radii of the resulting shell 
obtained  from~\eqref{TriangleIneq2}.

Proving the proposition turns out to be 
a little exercise in elementary normed vector space theory
rather exactly at the level of \cite[section 5.1]{Dieudonne},
and accordingly we have written the present sections~\ref{Sec.Old.Intro} and 
\ref{Proofs}  with a potential student of that reference  in mind.

\section{Proofs}\label{Proofs}

\begin{lem}\label{LemmaSrConnected}
Let $(E,\|\cdot\|)$ be a normed vector space over $\R$ with dimension at 
least $2$. Then, for each $r\in [0,\infty[$, the set $S_r$ from \eqref{Def Sr} 
is connected.
\end{lem}
\begin{proof}
Let $x,y\in E$ with $x\neq y$. Then we can choose    $z\in E$ linearly
independent of $y-x$, and observe that then
the continuous paths
\[
 [0,1]\,\,\ni\,\,t &\mapsto& x+ t\,(y-x) \,+\, 
 \min\{t,1-t\}\,\alpha \,z
 \,\, =:\,\, \gamma^{}_\alpha(t)
\]
with $\alpha \in \R$ have trajectories  disjoint except for their endpoints
$x$ and $y$,  that is, $\gamma^{}_\alpha({]0,1[}) \neq \gamma^{}_\beta({]0,1[})$
for $\alpha\neq \beta$.  Hence, if also $x,y\neq 0$, then
at most one  $\gamma_\alpha^{}$ passes through $0$. 
This shows that $E\setminus \{0\}$ is connected.  Hence so is its image  
$S_r$ under  the continuous map  $E\setminus \{0\} \ni x \mapsto rx/\|x\| $.
\end{proof}

\textsc{Remark.} The above proof yields in fact that  $E\setminus  A$ is 
connected whenever $A$ is at most countable. See 
\cite[exercise 5 in section 5.1]{Dieudonne} for further such results.

\begin{lem}\label{LemmaUnionOfIntervals} 
Let $I\subset \R$ be a nonempty compact interval and 
let $f,g:I\rightarrow \R$ be continuous functions with $f\le g$.
Then $\bigcup_{x\in I} [f(x),g(x)] = [\min f,\max g]$.
\end{lem}
\begin{proof}  
Let $A$ and $B$ denote the two sets claimed to be equal.
Trivially, $A\subset B$. Let $x,y\in I$ with $f(x)=\min f$ 
and $g(y)=\max g$. Then
\[
 z(t) &:=& \left\{
 \begin{array}{ll}
 f(x)+2t\big(g(x)-f(x) \big) &\quad
  \text{ for }\,\,t\in[0,\textstyle{\frac 12}]\\
  g\big(x+2(t-\textstyle{\frac 12})(y-x) \big) &\quad
  \text{ for }\,\,t\in {]\textstyle{\frac 12},1]} 
 \end{array}\right.
\]   
defines a continuous function $z:[0,1]\rightarrow A$, with 
$z(0)= \min f$ and $z(1)=\max g$. Hence, by the 
intermediate value theorem, we have $B\subset A$.   
\end{proof}

\begin{proof}[Proof of the Proposition]
Let $A$ and $B$ denote the two sides of \eqref{SumsAreShells1}. 
We have $A \subset B$ by \eqref{TriangleIneq1}.  
So let $z_0 \in B$. 
Using the notation from \eqref{Def Sr}, we consider the continuous
function  
\[
 S_a \,\,\ni\,\, x &\mapsto&  \|z_0-x\|  \,\,=:\,\, h(x)
\]
If $z_0 \neq 0$, we put $u:=z_0/\|z_0\|$,
otherwise we choose an arbitrary $u\in E$ with $\|u\|=1$.
Then $x_1 := a u$ and $x_2 := -au$ belong to $S_a$, and the 
assumption $|a-b|\le  \|z_0\| \le a+b$ yields
$h(x_1) = \|z_0 -au\| = |\|z_0\|-a|  \le b$
and $h(x_2) = \| z_0+au \|  = \|z_0 \| +a \ge b$.
By Lemma~\ref{LemmaSrConnected} and by the continuity of $h$, 
the image  $h(S_a)$
is connected, so there exists an $x\in S_a$ with
$h(x)=b$, that is, $y:= z_0-x  \in S_b$. Thus
$z_0 = x+y \in A$, 
and \eqref{SumsAreShells1} is proved.

If $n=2$, then \eqref{SumsAreShells2} is just
\eqref{SumsAreShells1} with $a_1,a_2$ in place of $a,b$.
So assume  that  $n\ge 2$, $a_1,\ldots,a_{n+1}\in[0,\infty[$,
and that \eqref{SumsAreShells2} holds. % as stated. 
To prepare for an application of  Lemma~\ref{LemmaUnionOfIntervals},
we let 
$\alpha:= \max_{k=1}^n (a_k-\sum_{i\in\{1,\ldots,n\}\setminus\{k\}}a_i)^{}_+$,
$\beta := \sum_{i=1}^n a_i$, and  
$f(r):=|r-a_{n+1}| $ and $g(r):= r+a_{n+1}$
for $r \in [\alpha, \beta]$, and observe that 
\[
 \min f &=& \left\{\begin{array}{ll} 
  0 &\text{ if }\alpha\le a_{n+1} \le \beta \\
  \alpha-a_{n+1} &\text{ if } a_{n+1} \le \alpha \\
  a_{n+1} - \beta & \text{ if }  a_{n+1} \ge \beta
 \end{array} \right\}
 \,\,=\,\, \max_{k=1}^{n+1} 
   \big(a_k-\sum_{i\in\{1,\ldots,n+1\}\setminus \{k\}}a_i\big)^{}_+
\]
and $\max g = \sum_{i=1}^{n+1} a_i$. 
Using now the  notations $S_a+S_b$
and $\sum_{i=1}^n S_{a_i}$ for the left hand sides of 
\eqref{SumsAreShells1} and \eqref{SumsAreShells2},
we may write the inductive hypothesis \eqref{SumsAreShells2} as
$\sum_{i=1}^n S_{a_i} =\bigcup_{r\in[\alpha,\beta]} S_r$ and obtain
\[
 \sum_{i=1}^{n+1} S_{a_i} %&=& \sum_{i=1}^n S_{a_i} \,+\, S_{a_{n+1}} \\
  &=& \bigcup_{r\in[\alpha,\beta]} S_r  \,\,+\,\, S_{a^{}_{n+1}} \\
  &=& \bigcup_{r\in[\alpha,\beta]}\big(\, S_r  \,+\, S_{a^{}_{n+1}}\,\big) \\
  &=& \bigcup_{r\in[\alpha,\beta]} \bigcup_{\varrho\in [f(r),g(r)]} S_\varrho
 \,\,\qquad[\text{by  \eqref{SumsAreShells1}}]\\
  &=&  \bigcup_{\varrho\in[\min f,\max g]}S_\varrho
  \qquad \qquad[\text{by Lemma~\ref{LemmaUnionOfIntervals}}]
\]
and the last set equals the right hand side of \eqref{SumsAreShells2}
with $n+1$ in place of $n$.
\end{proof}

% \section*{Acknowledgements}
% We thank ...

\end{document}